\documentclass[12pt]{amsart}
\usepackage{graphicx, soul}
\usepackage{color}
\usepackage{subfig}
\usepackage{float}
\usepackage{wasysym}
\usepackage{amssymb}
\usepackage{amsmath,amsthm,enumerate}
\usepackage{mathtools}

\DeclareMathOperator{\RE}{Re} 

 \def \D{\mathbb{D}}
 \def \C{\mathbb{C}}
 \def  \E{e^{i \theta}}
 \def \a{\alpha}
 \def \r {\mathcal{R}}
 \def \om{\omega}

\numberwithin{equation}{section}
\newtheorem{theorem}{Theorem}[section]
\newtheorem{lemma}[theorem]{Lemma}
\newtheorem{corollary}[theorem]{Corollary}
\theoremstyle{remark}
\newtheorem{remark}[theorem]{Remark}

\begin{document}

\title[Estimates of Normalized Analytic Functions]{Certain Estimates of Normalized Analytic Functions }

\author[S. Anand]{Swati Anand}

\address{Department of Mathematics, University of Delhi,
Delhi--110 007, India}
\email{swati\_anand01@yahoo.com}
\author [N.K. Jain]{Naveen Kumar Jain}
\address {Department of Mathematics, Aryabhatta College, Delhi-110021,India}
\email{naveenjain05@gmail.com}
\author [S. Kumar]{Sushil Kumar}
\address {Bharati Vidyapeeth's college of Engineering, Delhi-110063, India}
\email{sushilkumar16n@gmail.com}

\begin{abstract}
Let $\phi$ be a normalized convex function defined on open unit disk $\mathbb{D}$.
For a unified class of normalized analytic functions which satisfy the second order differential subordination $f'(z)+ \a z f''(z) \prec \phi(z)$ for all $z\in \mathbb{D}$, we investigate the distortion theorem and  growth theorem. Further, the bounds on initial logarithmic coefficients, inverse coefficient and the second Hankel determinant involving the inverse coefficients are examined.
\end{abstract}



\maketitle

\section{Introduction}
The class of all normalized analytic functions
\begin{equation} \label{fz}
f(z) = z +a_2 z^2+a_3 z^3 + \cdots
\end{equation}
in the unit disk $\D := \{z \in \C : |z| < 1\}$ is denoted by $\mathcal{A}$. Denote by $\mathcal{S}$,  the subclass of $\mathcal{A}$ consisting of univalent functions in $\mathbb{D}$. Let $\mathcal{P}$ be the class of analytic functions $p$ defined on $\D$, normalized by the condition $p(0) = 1$ and satisfying $\RE (p(z))>0$. Let $f$ and $g$ be analytic in $\D$. Then $f$ is subordinate to $g$, denoted by $f \prec g$, if there exists an analytic function $w$ with $w(0) =0$ and $|w(z)| < 1$ for $ z\in \D$ such that $f(z) = g(w(z))$. In particular, if $g$ is univalent in $\D$ then $f$ is subordinate to $g$ when $f(0) = g(0)$ and $f(\D) \subseteq g(\D)$. In this paper we shall assume that $\phi$ is an analytic function with positive real part in $\D$ and normalized by the conditions $\phi(0) = 1$ and $\phi'(0) > 0$. It is noted that $\phi(\D)$ is convex. The function $\phi$ is symmetric with respect to the real axis and it maps  $\D$ onto a region starlike with respect to $\phi(0) = 1$. The Taylor series representation of the function $\phi$ is given by
\begin{equation} \label{phiz}
\phi(z) = 1+B_1 z+B_2 z^2+B_3 z^3+ \cdots,
\end{equation}
where $B_1 > 0$. For such $\phi$,
Ma and Minda \cite{MA} studied the unified subclasses $S^*(\phi)$ and $C(\phi)$ of starlike and convex functions respectively, analytically defined as
\[S^*(\phi) = \left\{f \in \mathcal{A}: \frac{z f'(z)}{f(z)} \prec \phi(z)\right\}\,\, \text{and}\,\quad
C(\phi) = \left\{f \in \mathcal{A}: 1+\frac{z f''(z)}{f'(z)} \prec \phi(z)\right\}.\]
The authors investigated the growth, distortion and coefficient estimates for these classes. In particular, the class $S^*(\phi)$ reduces to some well known subclasses of starlike fuctions. For example, when $-1 \le B < A \le 1$, $S^*[A,B]$ is the class of Janowski starlike functions introduced by Janowski \cite{Jan}. For $0 \leq \alpha <1$, $S^*[1-2 \alpha,-1] = S_{\alpha}^*$ is the class of starlike functions of order $\alpha$, introduced by Robertson \cite{ROB}. The class $\mathcal{SL}=\mathcal{S^*}(\sqrt{1+z})$, introduced by  Sok\'{o}\l\ and\  Stankiewicz \cite{sok}, consists of functions $f\in\mathcal{A}$ such that $zf'(z)/f(z)$ lies in the region $\Omega_L:=\{w\in\mathbb{C}:|w^2-1|<1\}$, that is, the right-half of the lemniscate of Bernoulli. Later, Mendiratta et al. \cite{men} introduced the class $\mathcal{S^*}_e:=\mathcal{S^*}(e^z)$ consists of functions $f\in\mathcal{A}$ satisfying the condition $|\log(zf'(z)/f'(z))|<1$. In 2011, Ali  \emph{et al.} \cite{Ali2}(see also \cite{jain}) studied the class of all those functions $f \in \mathcal{A}$ which satisfy the following third order differential equation
\[f(z)+ \a f'(z) + \gamma z^2 f''(z) = g(z),\]
where the function $g$ is subordinate to a convex function $h$. In \cite{Ali2},  the best dominant on all solutions of the differential equation in terms of double integral were obtained. Some certain variations of the class
$\mathcal{R}(\a ,h) = \{f \in \mathcal{A}: f'(z) + \a z f''(z) \prec h(z), z \in \D\}$,
where $h$ is a convex function have been investigated by several authors \cite{gao,Nasr,SS,HM,Yang}.
On the basis of the above discussed works, we consider a unified class of all functions $f \in \mathcal{A}$ such that

 \begin{equation}
 f'(z) + \alpha z f''(z) \prec \phi(z)
 \end{equation}
for $z \in \D$ and where $\a \in \C$ with $\RE \a \ge 0$. The class of such functions is denoted by  $\r (\a , \phi)$.
Since $f \in R(\alpha, \phi)$,  $f'(z) + \alpha z f''(z) \ne \phi(\E)$, $\theta \in [0,2 \pi)$, it is observed that
\begin{equation*}\label{eq 9}
f'(z) + \alpha z f''(z) = [(1 - \alpha)f(z) + \alpha (zf'(z))]'.
\end{equation*}
Also, we have
$z f'(z) = f(z) \ast \dfrac{z}{(1-z)^2}$ and $f(z) = f(z) \ast \dfrac{z}{1-z}$.
Thus, 
\begin{align*} \label{eq 10}
\notag f'(z) + \alpha z f''(z) &= \bigg((1 - \alpha)f(z)\ast \frac{z}{1-z} + \alpha f(z)\ast \frac{z}{(1-z)^2}\bigg)'\\
& = \bigg( f(z) \ast \left((1-\alpha) \frac{z}{1-z} + \alpha \frac{z}{(1-z)^2}\right)\bigg)'.
\end{align*}
Therefore, we conclude that
\begin{equation*}
\bigg(f(z) \ast \frac{z - z^2(1-\alpha)}{(1-z)^2} \bigg)' - \phi(\E) \ne 0
\end{equation*}
or equivalently
\begin{equation*}
 \frac{1}{z}\bigg(f(z) \ast \frac{z + z^2(2 \alpha - 1)}{(1-z)^3} \bigg) \ne \phi(\E)
\end{equation*}
which is the necessary and sufficient conditions for a function $ f \in \mathcal{A}$ to be in the class $\r(\a ,\phi)$.

In this paper, we compute the distortion, growth inequalities for a function $f$ in the class $ \r(\a ,\phi)$. The sharp bounds on initial logarithmic coefficients for such functions are also obtained. Next, we obtain the bounds on initial inverse coefficients of the function $f \in  \r(\a ,\phi)$ as well as bounds on Fekete Szeg\"{o} functional and second Hankel determinant.

\section{Distortion and Growth Theorem}
The first theorem proves the distortion theorem of the functions $f$ belonging to $\r(\a ,\phi)$.
\begin{theorem} \label{DT}
Let $\alpha \in \C$,  $\RE \alpha \ge 0$ and the function $\phi$ be defined by \eqref{phiz}. If the function $f \in \r(\a ,\phi)$,  then
\[ 1+ \sum_{n = 1}^{\infty} \frac{|B_n| (-r)^n}{n\RE \alpha+1} \le |f'(z)| \le 1 + \sum_{n = 1}^{\infty} \frac{|B_n| r^n}{n \RE\alpha+1} \]
for $|z| < r < 1$. The result is sharp.
\end{theorem}
We make use of the following lemma in order to prove some of our results.
\begin{lemma} \cite[Lemma 2, p. 192]{HAL}\label{lem}
Let $h$ be a convex function  with $Re\, \gamma \ge 0$. If $p(z)$ is regular in $\D$ and $p(0) = h(0)$, then
\begin{equation}\label{eq 1l}
p(z) +\frac{zp'(z)}{\gamma} \prec h(z)
\end{equation}
implies that $p(z) \prec q(z) \prec h(z)$, where
\[q(z) = \gamma z^{-\gamma} \int_{0}^{z}h(t) t^{\gamma-1}dt.\]
The function $q$  is convex and the  best dominant.
\end{lemma}

\begin{proof}[Proof of Theorem \ref{DT}]
Let the fuction $f$ be in the class $\r( \a, \phi)$.  Then
$f'(z) + \alpha z f''(z) \prec \phi(z).$ For  $p(z) = f'(z)$  and $\gamma=1/\alpha$, Lemma \ref{lem} yields
\begin{equation*} \label{eq 2}
f'(z) \prec \frac{1}{\alpha} z^{-1/\alpha} \int_{0}^{z} \phi(t) t^{1/\alpha - 1} dt.
\end{equation*}
 On taking $t = z \zeta^{\alpha}$ in (\ref{eq 2}), we get
 \begin{equation}\label{eq 3}
 f'(z) \prec \int_{0}^{1} \phi(z \zeta^{\alpha})d \zeta.
 \end{equation}
Since the function $\phi$ is symmetric with respect to real axis, $\phi(z)$ has real coefficients. Also $\phi'(0)>0$ gives $\phi'(x)$ is increasing on $(0, 1)$. Thus, \begin{equation}\label{eq6}
\min_{|z=r|}\RE\phi(z)=\phi(-r)\quad \text{and}\quad \max_{|z=r|}\RE\phi(z)=\phi(r)
.\end{equation} Using \eqref{eq6} and \eqref{eq 3} for $|z|=r$, we have
\begin{align}\label{eq1}
|f'(z)|\geq\RE f'(z)\geq\min_{|z|=r}\RE f'(z)\geq\min_{|z|=r}\RE \int_{0}^1\phi(z \zeta^{\alpha})d \zeta&=\int_{0}^1\min_{|z|=r}\RE\phi(z \zeta^{\alpha})d \zeta\\ \nonumber
&=\int_{0}^1\phi(-r \zeta^{\RE \alpha})d \zeta.
\end{align}
Similarly, we have
\begin{equation}\label{eq2}
|f'(z)|\leq\int_{0}^1\phi(r \zeta^{\RE \alpha})d \zeta.
\end{equation}
Since $\phi(z \zeta^{\RE \alpha}) = 1 +B_1 z \zeta^{\RE \alpha} + B_2 z^2 \zeta^{2 \RE \alpha} + \cdots$, a simple calculation yields
\begin{align} \label{eq 4}
\notag \int_{0}^{1}\phi(z \zeta^{\RE \alpha}) d \zeta &= \int_{0}^{1}(1 +B_1 z \zeta^{\RE \alpha} + B_2 z^2 \zeta^{2 \RE \alpha} +B_3 z^3 \zeta^{3 \RE \alpha} + \cdots) d \zeta\\
\notag & = 1+\frac{B_1 z}{\RE \alpha +1} + \frac{B_2 z^2}{2 \RE \alpha +1}+\frac{B_3 z^3}{3 \RE \alpha +1}+ \cdots \\
& = 1+ \sum_{n =1 }^{\infty} \frac{B_n z^n}{n \RE \alpha +1}.
\end{align}
From \eqref{eq1}, \eqref{eq2} and \eqref{eq 4} the result follows.
The result is sharp for the function $f:\mathbb{D}\rightarrow\mathbb{C}$ defined by
\begin{equation}\label{eq11}
f(z)=z+ \sum_{n=1}^{\infty} \frac{B_n z^{n+1}}{(n+1)(1+n \RE\alpha)}.\qedhere
\end{equation}
\end{proof}

\begin{theorem}\label{th1}
Let $\alpha \in \C$ such that $\RE \alpha \ge 0$ and the function $\phi$ be as in (\ref{phiz}). Then for the function $f \in R(\alpha, \phi)$, we have
\[ 1+ \sum_{n = 1}^{\infty} \frac{|B_n| (-r)^n}{(n+1)(n \RE\alpha+1)} \leq \frac{|f(z)|}{|z|} \le 1 + \sum_{n = 1}^{\infty} \frac{|B_n| r^n}{(n+1)(n \RE\alpha+1)}\,\, (|z| < r < 1).\]

\end{theorem}

\begin{proof}
Let
\begin{equation}\label{eq5}
H(z) = \int_{0}^{1} \phi(z \zeta^{\alpha}) d \zeta
\end{equation} and
\[\Phi_{\alpha}(z) = \int_{0}^{1} \frac{1}{1-z t^{\alpha}}dt = \sum_{n = 0}^{\infty} \frac{z^n}{1+n \alpha}.\]
From \cite[Theorem 5, p.113]{RUS}, it is noted that $\Phi_{\alpha}(z)$ is convex with $\RE \a  \ge 0$.
Also,
\[\Phi_{\alpha}(z) \ast \phi(z)  = \left( \sum_{n = 0}^{\infty} \frac{z^n}{1+n \alpha}\right)* \left(1+ \sum_{n = 1}^{\infty} B_n z^n\right)
 = 1+  \sum_{n = 1}^{\infty} \left(\frac{B_n}{1+n \alpha}\right) z^n.
\]
It view of above and   \eqref{eq 4}, we have
\begin{equation*}
\Phi_{\alpha}(z) \ast \phi(z)
 = \int_{0}^{1}\phi(z \zeta^{\alpha})d \zeta = H(z).
\end{equation*}
Since convolution of two convex functions is convex, the function  $H$ is convex and $H(0) = 1$.
Putting $\gamma = 1$ and $h(z) = H(z)$ in Lemma \ref{lem}, we get
\begin{equation}\label{eq 6}
p(z) \prec \frac{1}{z} \int_{0}^{z}H(t)dt
\end{equation}
Using \eqref{eq5},  substituting $t = z\sigma$  and $p(z) ={f(z)}/{z}$ in \eqref{eq 6}, we have
\begin{equation}\label{eq7}
\frac{f(z)}{z}  \prec \int_{0}^{1} \int_{0}^{1} \phi(z\sigma \zeta ^{\alpha})d\sigma d \zeta.
\end{equation}
Let $h(z)=f(z)/z$. Then \eqref{eq6} together with \eqref{eq7} yields

\begin{equation}\label{eq8}
|h(z)|\geq\min_{|z|=r}\RE h(z)\geq\int_{0}^{1} \int_{0}^{1}\min_{|z|=r} \phi(z\sigma \zeta ^{\alpha})d\sigma d \zeta=\int_{0}^{1}\int_{0}^1\phi(-r \sigma\zeta^{\RE \alpha})d\sigma d \zeta.
\end{equation}
and
\begin{equation}\label{eq9}
|h(z)|\leq\max_{|z|=r}\RE h(z)\leq\int_{0}^{1} \int_{0}^{1}\max_{|z|=r} \phi(z\sigma \zeta ^{\alpha})d\sigma d \zeta=\int_{0}^{1}\int_{0}^1\phi(r \sigma\zeta^{\RE \alpha})d\sigma d \zeta.
\end{equation}
A simple calculation shows that
\begin{align}\label{eq10}
\int_{0}^{1} \int_{0}^{1} \phi(z\sigma \zeta ^{\RE\alpha})d\sigma d \zeta
 & = \int_{0}^{1} \left(1+ \sum_{n=1}^{\infty} \left(\frac{B_n}{1+n \RE\alpha}\right)(z\sigma)^n\right) d\sigma \nonumber \\ 
& = 1+ \sum_{n=1}^{\infty} \frac{B_n z^n}{(n+1)(1+n \RE\alpha)}.
\end{align}
Now, the result follows from \eqref{eq8}, \eqref{eq9} and \eqref{eq10}.
The result is sharp for the function $f$ given by \eqref{eq11}.
\end{proof}

\begin{remark}
Letting $\phi(z) = (1-(1-2 \beta)z)/(1-z)$, where $\beta <1$ and $\alpha = 1$, Theorem \ref{th1}  reduces to a result due to Silverman \cite[Corollary 2, p.250]{SILVERMAN}.
Further, for $\phi(z) = (1-(1-2 \beta)z)/(1-z)$, where $\beta <1$ and $\alpha >0$, Theorem \ref{CI} yields \cite[Corollary 3, p.178]{gao}.
\end{remark}

\begin{theorem} \label{CI}
Let $\alpha \in \C$ such that $\RE \alpha \ge 0$ and the function $\phi$ be as in (\ref{phiz}) such that  $f \in R(\alpha, \phi)$. Then
\begin{align}\label{2,1}
|a_n| \le \frac{B_1}{|n + n(n-1) \alpha|}, \quad  \text{for all} \quad n \ge 2.
\end{align}

\end{theorem}


\begin{proof}
For $p(z) = 1+ \sum_{n = 1}^{\infty}p_n z^n \in \mathcal{P}$, set
$f'(z) + \alpha z f''(z) = p(z)$,  $ z \in \D$. Since $f \in R(\alpha, \phi)$,  $p(z) \prec \phi(z)$.
A simple calculation gives
\begin{equation} \label{eq 8}
 f'(z) + \alpha z f''(z)
 = 1 + \sum_{n= 2}^{ \infty} [n +n(n-1) \alpha] a_n z^{n-1} = 1+ \sum_{n = 1}^{\infty}p_n z^n.
\end{equation}
On comparing the coefficients of $z^{n-1}$, we get
\[(n + n(n-1) \alpha) a_n = p_{n-1}, \quad \text{for all} \quad n \ge 2.\]
By making use of \cite[Theorem X, p.70] {ROG}, we get $|p_n| \le B_1$, for all $n \ge 1.$
Hence, we get the desired inequality.
The inequality \eqref{2,1} is sharp for the function $f_n$ given by
$f_n'(z) + \alpha z f_n ''(z) = \phi(z^{n-1}).$
\end{proof}
\begin{remark}
On taking $\phi(z) = (1-(1-2 \beta)z)/(1-z)$, where $\beta <1$ and $\alpha = 1$, Theorem \ref{CI}  yields  a result due to Silverman \cite[Corollary 1, p.250]{SILVERMAN}.
Further, letting $\phi(z) = (1-(1-2 \beta)z)/(1-z)$, where $\beta <1$ and $\alpha >0$ Theorem \ref{CI} reduces to \cite[Corollary 2, p.178]{gao}.
\end{remark}

%
%

\section{Bounds on Initial Logarithmic Coefficient}
For a function $f \in \mathcal{S}$, the logarithmic coefficients $\gamma_n$ are defined by the following series expansion:
\begin{equation} \label{4.1}
\log \frac{f(z)}{z} = 2 \sum\limits_{n =1}^{\infty} \gamma_n z^n ,\, z \in \mathbb{D} \setminus \{0\}, \log 1 := 0.
\end{equation}
On comparing the coefficients of $z$ on both the sides, we get the initial logarithmic coefficients
\begin{align}
&\gamma_1 = \frac{1}{2}a_2, \quad\gamma_2 = \frac{1}{2}\left(a_3 - \frac{1}{2}{a_2}^2\right) \notag\\
&\gamma_3 = \frac{1}{2}\left(a_4 - a_2a_3 + \frac{1}{3}{a_2}^3\right) \label{gam}
\end{align}
In 1979, the authors \cite{DUREN} showed that the logarithmic coefficients $\gamma_n$ of every function $f\in \mathcal{S}$ satisfy the inequality
$\sum\limits_{n=1}^{\infty}|\gamma_n|^2 \le {\pi ^2}/{6}$, where the equality holds if and only if the function $f$ is rotation of the Koebe function $k(z) = z (1-\E)^{-2}$ for each $\theta$. The $n^{th}$ logarithmic coefficient of $k(z)$ is $\gamma_n = e^{i n \theta}/n$ for each $\theta$ and $n \ge 1$. In \cite{Ye}, the logarithmic coefficients $\gamma_n$ of each close-to-convex function $f \in S$ is bounded by $(A \log n)/n$ where $A$ is an absolute constant. In 2018, the authors \cite{MFALI,PK} obtained the bounds on logarithmic coefficients of certain subclasses of the class of close-to-convex functions. Recently, Adegani \emph{et. al.} \cite{ADEGANI} investigated the bounds for the  initial logarithmic coefficients of the generalized classes $S^*(\phi)$ and $C(\phi)$.
To find the bounds on initial logarithmic coefficient for class $\mathcal{R}(\alpha, \phi)$, we shall use the following two lemmas.

\begin{lemma}\cite[p.172] {NEH}\label{l1}
Assume that $w$ is a Schwarz function so that $w(z) = \sum\limits_{n =1}^{\infty}c_n z^n$. Then
\[|c_1| \le 1\qquad \mbox{and} \qquad |c_n| \le 1-|c_1|^2, \,\, n= 2,3, \cdots .\]
\end{lemma}

\begin{lemma} \cite [Theorem 1] {DV} \label{lemH}
Let $w(z) = \sum\limits_{n=1}^{\infty}c_n z^n$ be the Schwarz function. Then for any real numbers $q_1$ and $q_2$, the following sharp inequality holds:
\[|c_3+q_1c_1c_2+q_2 c^3_1| \le H(q_1;q_2),\]
where
\[ H(q_1;q_2) =  \left \{ \begin{array}{ll}
                1, & \mbox{if $(q_1,q_2) \in D_1 \cup D_2 \cup \{(2,1)\}$,} \\
                  |q_2| , & \mbox{if $(q_1,q_2) \in \cup_{k=1}^{7}D_{k} $,}\\
                  \dfrac{2}{3}(|q_1|+1)\left(\dfrac{1+|q_1|}{3(|q_1|+1+q_2)}\right)^\frac{1}{2}, & \mbox{if $(q_1,q_2) \in D_8 \cup D_9,$}\\
                  \dfrac{q_2}{3}\left(\dfrac{q^2_1 - 4}{q^2_1-4q_2}\right) \left(\dfrac{q^2_1 -4}{3(q_2-1)}\right)^\frac{1}{2}, & \mbox{if $(q_1,q_2) \in D_{10} \cup D_{11} \setminus \{(2,1)\} ,$}\\
                  \dfrac{2}{3}(|q_1|-1)\left(\dfrac{|q_1|-1}{3(|q_1|-1-q_2)}\right)^\frac{1}{2}, & \mbox{if $(q_1,q_2) \in D_{12}$}
                             \end{array}
                       \right.\]
 and for $k =1,2, \cdots 12,$ the sets $D_k$ are defined as follows:
 \begin{align*}
 D_1 &= \left\{(q_1,q_2): |q_1| \le \frac{1}{2},|q_2| \le 1\right\},\\
 D_2 &= \left\{(q_1,q_2): \frac{1}{2} \le|q_1| \le 2,\frac{4}{27}(|q_1|+1)^3- (|q_1|+1)\le  q_2 \le 1 \right\},\\
D_3 &= \left\{(q_1,q_2): |q_1| \le \frac{1}{2},q_2 \le -1\right\},\\
D_4 &= \left\{(q_1,q_2): |q_1| \ge \frac{1}{2},q_2 \le -\frac{2}{3}(|q_1|+1) \right\},\\
 D_5 &= \left\{(q_1,q_2): |q_1| \le 2,q_2 \ge 1\right\},\\
 D_6 &= \left\{(q_1,q_2): 2 \le |q_1| \le 4,q_2 \ge \frac{1}{12}(q_1^2+8) \right\},\\
  D_7 &= \left\{(q_1,q_2): |q_1| \ge 4,q_2 \ge \frac{2}{3}(|q_1|-1)\right \},\\
 D_8 &= \left\{(q_1,q_2): \frac{1}{2} \le|q_1| \le 2,-\frac{2}{3}(|q_1|+1) \le q_2 \le \frac{4}{27}(|q_1|+1)^3- (|q_1|+1)\right  \},
\end{align*}
\begin{align*}
D_9 &=\left \{(q_1,q_2): |q_1| \ge 2,-\frac{2}{3}(|q_1|+1) \le q_2 \le \dfrac{2|q_1|(|q_1|+1)}{q^2_1+2|q_1|+4} \right \},\\
D_{10} &= \left \{(q_1,q_2): 2 \le |q_1| \le 4, \dfrac{2|q_1|(|q_1|+1)}{q^2_1+2|q_1|+4} \le q_2  \le \frac{1}{12}(q_1^2+8)\right \},\\
D_{11} &= \left \{(q_1,q_2): |q_1| \ge 4, \dfrac{2|q_1|(|q_1|+1)}{q^2_1+2|q_1|+4} \le q_2  \le \dfrac{2|q_1|(|q_1|-1)}{q^2_1-2|q_1|+4}\right \},\\
D_{12} &= \left \{(q_1,q_2):  |q_1| \ge 4, \dfrac{2|q_1|(|q_1|-1)}{q^2_1-2|q_1|+4} \le q_2  \le \frac{2}{3}(|q_1|-1)\right \}.
  \end{align*}

\end{lemma}

\begin{theorem} \label{LOG}
Let $\alpha \in \C$ such that $\RE \alpha \ge 0$ and the function $\phi$ be as in (\ref{phiz}). Suppose $f \in \mathcal{R}(\alpha, \phi)$, then the initial  logarithmic coefficients of $f$ satisfy the following inequalities:
\begin{itemize}
\item[(i)] $|\gamma_1| \le \dfrac{B_1}{4|1+\a|}$.
\item [(ii)] \[|\gamma_2| \le \left\{ \begin{array}{ll}
                \dfrac{B_1}{6 |1+2 \a|}, & \mbox{if $|8(1+\a )^2B_2 -3(1+2 \a )B^2_1| \le 8 B_1 |(1+\a )^2|$} \\
                   \dfrac{|8(1+\a )^2B_2 -3(1+2 \a )B^2_1|}{48|(1+\a )^2| |1+ 2\a |} , & \mbox{if $|8(1+\a )^2B_2 -3(1+2 \a )B^2_1| > 8 B_1 |(1+\a )^2|$}
                             \end{array}
                       \right.\]
\item[(iii)] If $B_1,B_2,B_3$ and $\a $ are real numbers, then
\[|\gamma_3| \le \frac{B_1}{8(1+3 \a )} H(q_1; q_2),\]
where $H(q_1; q_2)$ is given in Lemma \ref{lemH} such that
\[q_1 = \frac{2B_2}{B_1} - \frac{2B_1(1+3 \a )}{3(1+ \a )(1+2 \a )}\] and
\[q_2 = \frac{B_3}{B_1} - \frac{2 B_2 (1+3 \a )}{3 (1+\a )(1+2 \a )} + \frac{B^2_1(1+3 \a )}{6(1+ \a )^3}.\]
\end{itemize}
The bounds for $\gamma_1$ and $\gamma_2$ are sharp.
\end{theorem}

\begin{proof}
Let $f(z) = z+ \sum\limits_{n=2}^{\infty}a_n z^n\in \r (\a , \phi)$ where $\phi$ is given by (\ref{phiz}). Then $f'(z) + \a z f''(z) = \phi(w(z))$
for $z \in \D$, where $w(z) = \sum\limits_{n=1}^{\infty}c_n z^n$ is the Schwarz function.
A simple calculation yields
\[f'(z) + \a z f''(z) = 1+B_1 c_1 z+ (B_1 c_2+B_2 c^2_1)z^2 + (B_1c_3 + 2c_1 c_2 B_2+B_3 c^3_1)z^3+ \cdots .\]
On comparing the coefficients of $z$, we obtain
\begin{align}
2(1+\a )a_2 &= B_1c_1,\notag\\
3(1+2\a ) a_3 &= B_1c_2+B_2c^2_1  \quad\text{and} \notag\\
4(1+3 \a ) a_4 &= B_1c_3 +2 B_2 c_1 c_2  +B_3c^3_1. \nonumber
\end{align}
On substituting the above values of $a_i \,(i=1,2,3)$  in  \eqref{gam}, we get
\begin{align}
\gamma_1 &= \frac{B_1c_1}{4(1+\a )} , \notag \\
\gamma_2 &= \frac{8(1+\a )^2 B_1c_2+(8(1+\a )^2B_2 -3(1+2\a )B^2_1)c^2_1}{48(1+\a )^2 (1+2\a )} , \notag \\
\gamma_3 &= \frac{B_1}{8(1+3 \a )} c_3 +  \left(\frac{B_2}{4(1+3 \a )} - \frac{B^2_1}{12(1+
\a )(1+2 \a )}\right)c_1 c_2 \notag\\
& + \left(\frac{B_3}{8(1+3 \a )} - \frac{B_1 B_2}{12(1+
\a )(1+2 \a )} + \frac{B^3_1}{48(1+ \a )^3}\right)c^3_1. \label{ag}
\end{align}
By using Lemma \ref{l1}, we get the desired best possible estimate on $\gamma_1$. The bound is sharp for $|c_1| = 1$.

\begin{align*}
|\gamma_2| &= \left|\frac{B_1}{6(1+2\a )}c_2 + \frac{8(1+\a )^2B_2 -3(1+2\a )B^2_1}{48(1+\a )^2 (1+2\a )} c^2_1 \right|\\
& \le \frac{B_1}{6|1+2\a |}|c_2| + \frac{|8(1+\a )^2B_2 -3(1+2\a )B^2_1|}{48|(1+\a )^2| |1+2\a |} |c^2_1| \\
& \le \frac{B_1}{6|1+2\a |}(1-|c_1|^2) + \frac{|8(1+\a )^2B_2 -3(1+2\a )B^2_1)|}{48|(1+\a )^2| |1+2\a |} |c^2_1| \\
&=\frac{B_1}{6|1+2\a |}+ \left(\frac{|8(1+\a )^2B_2 -3(1+2\a )B^2_1|}{48|(1+\a )^2| |1+2\a |} - \frac{B_1}{6|1+2\a |}\right)|c^2_1| \\
&\le \left\{ \begin{array}{ll}
                \dfrac{B_1}{6 |1+2 \a|}, & \mbox{if $\dfrac{|8(1+\a )^2B_2 -3(1+2 \a )B^2_1|}{48|(1+\a )^2| |1+ 2\a |} \le \dfrac{B_1}{6 |1+2 \a|}$} \\
                   \dfrac{|8(1+\a )^2B_2 -3(1+2 \a )B^2_1|}{48|(1+\a )^2| |1+ 2\a |} , & \mbox{if $\dfrac{|8(1+\a )^2B_2 -3(1+2 \a )B^2_1|}{48|(1+\a )^2| |1+ 2\a |} > \dfrac{B_1}{6 |1+2 \a|}$}
                             \end{array}
                       \right.
\end{align*}
These bounds are sharp for $|c_1| = 0$ and $|c_1| = 1$, respectively.

The third inequality follows by Lemma \ref{l1}.
Using  Lemma \ref{lemH} for $\gamma_3$, we obtain
\begin{align*}
|\gamma_3| &= \left| \frac{B_1}{8(1+3 \a )} c_3 +  \left(\frac{B_2}{4(1+3 \a )} - \frac{B^2_1}{12(1+
\a )(1+2 \a )}\right)c_1 c_2\right. \\
&\left.+ \left(\frac{B_3}{8(1+3 \a )} - \frac{B_1 B_2}{12(1+
\a )(1+2 \a )} + \frac{B^3_1}{48(1+ \a )^3}\right)c^3_1 \right|\\
&=\frac{B_1}{8(1+3 \alpha)} |c_3 +c_1c_2q_1 +c^3_1q_2|\\
&\le \frac{B_1}{8(1+3 \alpha)} H(q_1;q_2),
\end{align*}
where $q_1 = 2 \left( \dfrac{B_2}{B_1} - \dfrac{B_1(1+3 \a )}{3(1+ \a )(1+2 \a )}\right)$ and $q_2 = \dfrac{B_3}{B_1} - \dfrac{2B_2(1+3\a )}{3(1+\a )(1+2 \a )} + \dfrac{B^2_1(1+3 \a )}{6(1+\a )^3}$.
This completes the proof.
\end{proof}
On taking $\phi(z) = e^z$, $\phi(z) = \sqrt{1+z}$ and $\phi(z) = (1+z)/(1-z)$, respectively in Theorem \ref{LOG} the following corollaries follows immediately.

\begin{corollary}
Let $\alpha \in \C$ such that $\RE \alpha \ge 0$ and $\phi(z) = e^z$. Suppose $f \in \mathcal{R}(\alpha, \phi)$, then the initial  logarithmic coefficients of $f$ satisfy the following inequalities:
\begin{itemize}
\item[(i)] $|\gamma_1| \le \dfrac{1}{4|1+\a|}$
\item [(ii)] $|\gamma_2| \le
                \dfrac{1}{6 |1+2 \a|}$
 \item[(iii)]
$|\gamma_3| \le \dfrac{1}{8(1+3 \a )} H(q_1; q_2)$,

where $H(q_1; q_2)$ is given in Lemma \ref{lemH} such that
\[q_1 = \frac{1+3 \a(1+2 \a)}{3(1+ \a )(1+2 \a )}\] and
\[q_2 = \frac{1}{6} - \frac{(1+3 \a )}{3 (1+\a )(1+2 \a )} + \frac{(1+3 \a )}{6(1+ \a )^3}.\]
\end{itemize}
\end{corollary}

\begin{corollary}
Suppose that $f \in \mathcal{R}(\alpha, \phi)$ where $\RE \alpha \ge 0$ and $\phi(z) = \sqrt{1+z}$, then the initial  logarithmic coefficients of $f$ satisfy the following inequalities:
\begin{itemize}
\item[(i)] $|\gamma_1| \le \dfrac{1}{8|1+\a|}$
\item [(ii)] $|\gamma_2| \le
                \dfrac{1}{12 |1+2 \a|}$
 \item[(iii)]
$|\gamma_3| \le \dfrac{1}{16(1+3 \a )} H(q_1; q_2)$,

where $H(q_1; q_2)$ is given in Lemma \ref{lemH} such that
\[q_1 = -\frac{1}{2}-\frac{(1+3 \a)}{3(1+ \a )(1+2 \a )}\] and
\[q_2 = \frac{1}{8} + \frac{1+3 \a }{12 (1+\a )(1+2 \a )} + \frac{1+3 \a }{24(1+ \a )^3}.\]
\end{itemize}
\end{corollary}

\begin{corollary}
Let the function $f \in \mathcal{R}(\alpha, \phi)$ where $\RE \alpha \ge 0$ and $\phi(z) = (1+z)/(1-z)$, then the initial  logarithmic coefficients of $f$ satisfy the following inequalities:
\begin{itemize}
\item[(i)] $|\gamma_1| \le \dfrac{1}{2|1+\a|}$
\item [(ii)] $|\gamma_2| \le
                \dfrac{1}{3 |1+2 \a|}$
 \item[(iii)]
$|\gamma_3| \le \dfrac{1}{4(1+3 \a )} H(q_1; q_2)$,

where $H(q_1; q_2)$ is given in Lemma \ref{lemH} such that
\[q_1 = \frac{2(1+3 \a(1+2 \a))}{3(1+ \a )(1+2 \a )}\] and
\[q_2 = 1 - \frac{4(1+3 \a )}{3 (1+\a )(1+2 \a )} + \frac{2(1+3 \a )}{3(1+ \a )^3}.\]
\end{itemize}
\end{corollary}

\section{Inverse Coefficient Estimates}
From the Koebe one quarter theorem, the image of $\D$ under a function $f\in \mathcal{S}$ contains a disk of radius $1/4$. Thus for every univalent function $f$ there exist inverse function $f^{-1}$ such that $f^{-1}(f(z)) = z$ for $z \in
\D$ and $f(f^{-1}(\omega)) = \omega$ for $ |\omega| < r_0(f)$ where $r_0(f) \ge 1/4$. The function $f^{-1}$ has the Taylor series expansion
$f^{-1}(\omega) = \omega +A_2 \omega ^2+A_3 \omega ^3+ \cdots$
in some neighborhood of origin.
In 1923, L\"{o}wner \cite{LOWNER} initiated the problem of estimating the coefficients of inverse function and  investigated the coefficient estimates for inverse function $f \in \mathcal{S}$.  Later on, this lead to the study of inverse coefficient problem for several subclasses of $\mathcal{S}$ by various authors \cite{ALIC,LIB1,LIB2,LIB3,DV}. In \cite{KAPOOR,KRZY}, authors obtained the initial inverse coefficients for the well known classes $C$ and $S^*(\a )$ $(0 \le \a \le 1)$. Recently, Ravichandran and Verma \cite{VS} determined the bounds on inverse coefficient for the Janowski starlike functions.

In this section, we shall investigate the bounds on inverse coefficient. The following lemma is needed to obtain the  coefficient bounds for the inverse function.


\begin{lemma}\cite [Lemma 3, p.254] {LIB1} \label{l4}
If $p(z) = 1+p_1z+p_2z^2+ \dots$ is a function in the class $\mathcal{P}$, then for any complex number $\nu$,
\[|p_2 - \nu p^2_1| \le 2 \max\{1,|2\nu-1|\}.\]
\end{lemma}

\begin{theorem} \label{INV}
Let $\alpha \in \C$ such that $\RE \alpha \ge 0$ and the function $\phi$ defined by (\ref{phiz}). If function $f(z) = z+ \sum\limits_{n=2}^{\infty}a_n z^n \in \r(\a , \phi)$ and $f^{-1}(\omega) = \om + \sum\limits_{n=2}^{\infty}A_n \om^n $ for all $\om$ in some neighbourhood of the origin, then
\begin{itemize}
\item[(i)] $|A_2|  \le \dfrac{B_1}{2|1+\a |}$,
\item[(ii)] $|A_3|  \le \dfrac{B_1}{3|1+2 \a|}  \max\{1,|\mu|\}$,  where $\mu = \dfrac{3B_1(1+2 \a)}{2(1+\a)^2}-\dfrac{B_2}{B_1}$
\item[(iii)] If $B_1,B_2,B_3$ and $\a $ are real numbers, then
\[|A_4| \le \frac{B_1}{4(1+3 \a )} H(q_1; q_2),\]
where $H(q_1; q_2)$ is given in Lemma \ref{lemH} such that
\[q_1 = 2 \left(\frac{B_2}{B_1} - \frac{5B_1(1+3 \a )}{3(1+ \a )(1+2 \a )}\right)\] and
\[q_2 = \frac{B_3}{B_1} - \frac{10 B_2 (1+3 \a )}{3 (1+\a )(1+2 \a )} + \frac{5 B^2_1(1+3 \a )}{2(1+ \a )^3}.\]
\end{itemize}
\end{theorem}

\begin{proof}
Let $f \in \r(\a , \phi)$. Then
\begin{equation} \label{fphi}
f'(z) + \a z f''(z) = \phi(w(z)),
\end{equation}
where $w(z)$ is the analytic function $w$ with $w(0) = 0$ and $|w(z)| <1$.
Let \[p(z) = \frac{1+w(z)}{1-w(z)} = 1+p_1z +p_2 z^2 +p_3 z^3+ \cdots.\]
Since $w : \D \to \D$ is analytic thus $p$ is a function with positive real part and
\begin{equation} \label{eq w(z)}
w(z) = \frac{p(z)-1}{p(z) +1} = \frac{1}{2} p_1z +\frac{1}{2}(p_2 - \frac{p^2_1}{2})z^2 + \frac{1}{8}(p^3_1-4p_1p_2+4p_3)z^3+ \cdots.
\end{equation}
Then
\begin{align}
\phi(w(z)) &= 1+\frac{B_1p_1}{2}z+ \left(\frac{1}{4}B_2p^2_1 +\frac{1}{2}B_1(p_2-\frac{1}{2}p^2_1)\right)z^2 \notag\\
&+ \frac{1}{8}( (B_1-2B_2+B_3)p^3_1 +4(-B_1+B_2)p_1 p_2 +4B_1p_3) z^3 + \cdots.\label{phi}
\end{align}
Using expressions \eqref{phi} and \eqref{fphi}, we obtain
\begin{align}
2(1+ \a ) a_2&=\frac{ B_1p_1}{2},\notag \\
3(1+2 \a )a_3 &=\frac{1}{4}B_2 p^2_1 + \frac{1}{2}B_1(p_2-\frac{1}{2}p^2_1)
\quad \text{and} \notag \\
4(1+3 \a )a_4 &=  \frac{1}{8}( (B_1-2B_2+B_3)p^3_1 +4(-B_1+B_2)p_1 p_2 +4B_1p_3).\label{aB}
\end{align}
Since $f^{-1}(\omega) = \om + A_2 \om^2 +A_3 \om^3+A_4 \om^4 +\cdots$ in some neighbourhood of origin, so we have $f(f^{-1}(\om)) = \om$. That is
\begin{align*}
\om &= f^{-1}(\om)+a_2 (f^{-1}(\om))^2 +a_3 (f^{-1}(\om))^3+ \cdots\\
& = \om + A_2 \om^2 +A_3 \om^3+A_4 \om^4 +\cdots +a_2(\om + A_2 \om^2 +A_3 \om^3+A_4 \om^4 +\cdots)^2\\
&+a_3(\om + A_2 \om^2 +A_3 \om^3+A_4 \om^4 +\cdots)^3.
\end{align*}
A simple calculation gives the following realtions:
\begin{align}
A_2 &= -a_2 ,\notag\\
A_3 &= 2a^2_2 -a_3 \quad \text{and}\notag\\
A_4 &=-5a^3_2 +5a_2a_3 -a_4.\label{Aa}
\end{align}
On subsituting the values of $a_i$ from (\ref{aB}) into (\ref{Aa}) and a simple calculation yields
\begin{align*}
A_2 & = \dfrac{-B_1}{4(1+\a )} p_1,\\
A_3 & = \dfrac{-B_1}{6ˆ(1+2 \a )}p_2 + \left(\dfrac{B^2_1}{8(1+\a )^2} - \dfrac{B_2}{12(1+2 \a)} +\dfrac{B_1}{12(1+2 \a)}\right)p^2_1\\
\end{align*}
In (\ref{eq w(z)}), on taking $c_1 = \dfrac{p_1}{2}$, $c_2 = \dfrac{1}{2}(p_2-\dfrac{p^2_1}{2})$ , $c_3 = \dfrac{1}{8} (p^3_1-4p_1p_2+4p_3)$ and so on we get,
\begin{align}
2(1+\a )a_2 &= B_1c_1,\notag\\
3(1+2\a ) a_3 &= B_1c_2+B_2c^2_1  \quad\text{and} \notag\\
4(1+3 \a ) a_4 &= B_1c_3 +2 B_2 c_1 c_2  +B_3c^3_1. \label{eq ai}
\end{align}
On substituting the values of $a_i$ from (\ref{eq ai}) in (\ref{Aa}), we obtain
\begin{align*}
A_4 & = \frac{-B_1}{4(1+3 \a )} c_3 +  \left( \frac{5 B^2_1}{6(1+
\a )(1+2 \a )} -\frac{B_2}{4(1+3 \a )} \right)c_1 c_2 \\
 &+ \left(\frac{-B_3}{4(1+3 \a )} + \frac{5 B_1 B_2}{6(1+
\a )(1+2 \a )} + \frac{-5 B^3_1}{8(1+ \a )^3}\right)c^3_1.
\end{align*}

Since $|p_1| \le 2$,  we have
\[|A_2| \le \frac{B_1}{2|1+\a |}. \]
Consider
\[ |A_3| = \frac{B_1}{6|1+2 \a|}\left|p_2 - \left(\frac{3B_1(1+2 \a)}{4(1+\a)^2} -\frac{B_2}{2B_1}+\frac{1}{2}\right)p_1^2\right|.\]
Then by Lemma \ref{l4}, we get the desired estimate.
Using  Lemma \ref{lemH} for $A_4$, we obtain
\begin{align*}
|A_4| &= \left| \frac{-B_1}{4(1+3 \a )} c_3 +  \left(\frac{-B_2}{2(1+3 \a )} + \frac{5B^2_1}{6(1+
\a )(1+2 \a )}\right)c_1 c_2\right. \\
&\left.+ \left(\frac{-B_3}{4(1+3 \a )} + \frac{5 B_1 B_2}{6(1+
\a )(1+2 \a )} - \frac{5B^3_1}{8(1+ \a )^3}\right)c^3_1 \right|\\
&=\frac{B_1}{4(1+3 \alpha)} |c_3 +c_1c_2q_1 +c^3_1q_2|\\
&\le \frac{B_1}{4(1+3 \alpha)} H(q_1;q_2),
\end{align*}
where $q_1 = 2 \left( \dfrac{B_2}{B_1} - \dfrac{5B_1(1+3 \a )}{3(1+ \a )(1+2 \a )}\right)$ and $q_2 = \dfrac{B_3}{B_1} - \dfrac{10B_2(1+3\a )}{3(1+\a )(1+2 \a )} + \dfrac{5B^2_1(1+3 \a )}{2(1+\a )^3}$.
\end{proof}

The following corollaries are an immediate consequence of the Theorem \ref{INV} for $\phi(z) = e^z$, $\phi(z) = \sqrt{1+z}$ and $\phi(z) = (1+z)/(1-z)$, respectively.

\begin{corollary}
Let $\alpha \in \C$ such that $\RE \alpha \ge 0$ and  $\phi(z)=e^z$. If the function $f(z) = z+ \sum\limits_{n=2}^{\infty}a_n z^n \in \r(\a , \phi)$ and $f^{-1}(\omega) = \om + \sum\limits_{n=2}^{\infty}A_n \om^n $ for all $\om$ in some neighbourhood of the origin, then
\begin{itemize}
\item[(i)] $|A_2|  \le \dfrac{1}{2|1+\a |}$,
\item[(ii)] $|A_3|  \le \dfrac{1}{3|1+2 \a|}  \max\left\{1,\left|\dfrac{3(1+2 \a)}{2(1+\a)^2}-\dfrac{1}{2}\right|\right\}$

\item[(iii)]$|A_4| \le \dfrac{1}{4(1+3 \a )} H(q_1; q_2)$,

where $\a$ is real and $H(q_1; q_2)$ is given in Lemma \ref{lemH} such that
\[q_1 = \frac{-7+3 \a(-7+2 \a)}{3(1+ \a )(1+2 \a )}\] and
\[q_2 = \frac{1}{6} - \frac{5(1+3 \a )}{3 (1+\a )(1+2 \a )} + \frac{5(1+3 \a )}{2(1+ \a )^3}.\]

\end{itemize}
\end{corollary}

\begin{corollary}
 Suppose that the function $f(z) = z+ \sum\limits_{n=2}^{\infty}a_n z^n \in \r(\a , \phi)$, where $\RE \alpha \ge 0$ and  $\phi(z)=\sqrt{1+z}$ and $f^{-1}(\omega) = \om + \sum\limits_{n=2}^{\infty}A_n \om^n $ for all $\om$ in some neighbourhood of the origin, then
\begin{itemize}
\item[(i)] $|A_2|  \le \dfrac{1}{4|1+\a |}$,
\item[(ii)] $|A_3|  \le \dfrac{1}{6|1+2 \a|}  \max\left\{1,\left|\dfrac{3(1+2 \a)}{4(1+\a)^2}+\dfrac{1}{4}\right|\right\}$,
\item[(iii)] $|A_4| \le \dfrac{1}{8(1+3 \a )} H(q_1; q_2)$,

where $\a$ is real and $H(q_1; q_2)$ is given in Lemma \ref{lemH} such that
\[q_1 = -\frac{1}{2}-\frac{5(1+3 \a)}{3(1+ \a )(1+2 \a )}\] and
\[q_2 = \frac{1}{8} + \frac{5(1+3 \a )}{12 (1+\a )(1+2 \a )} + \frac{5(1+3 \a) }{8(1+ \a )^3}.\]
\end{itemize}
\end{corollary}
%
%
\section{Hankel Determinant}
The problem involving coefficient bounds have attracted the interest of many authors in particular to second Hankel determinants and Fekete Szeg\"{o} functional. The coefficient functional $|a_3 - \mu a^2_2|$ where $\mu$ is a complex number is called the Fekete Szeg\"{o} functional. The Fekete Szeg\"{o} problem involves maximizing  the functional $|a_3 - \mu a^2_2|$. Some authors have investigated the Fekete Szeg\"{o} problem for the coefficients of inverse function \cite{ALIC,Naz20,THOMAS}.
The Hankel determinant $|H_2(1)| = |a_3-a^2_2|$ is a particular case of the Fekete Szeg\"{o} functional and $H_2(2) = |a_2a_4-a^3_2|$ is called the second Hankel determinants. In 2013, Lee \emph{et. al.}  \cite{LEE}  obtained the bounds for the second Hankel determinant for the unified class of Ma-Minda starlike and convex functions. The authors \cite{THOMAS}  estimated the bounds on second Hankel determinant for the class of strongly convex functions of order $\a $ using the inverse coefficients. One may refer to \cite{Kumar17,HAY,Procd20,HAYMAN,NOONAN} for more details.
In the present section, we shall determine the Fekete Szeg\"{o} functional for the inverse coefficient.

\begin{theorem}
Suppose $\alpha \in \C$ such that $\RE \alpha \ge 0$ and the function $\phi$ defined by (\ref{phiz}). Let $f \in  \r(\a , \phi)$ and $f^{-1}(\omega) = \om + \sum\limits_{n=2}^{\infty}A_n \om^n $ for all $\om$ in some neighbourhood of the origin. Then for any complex number $\mu$, we have
\[|A_3 - \mu A^2_2| \le \frac{B_1}{3|1+2 \a |} \max\left\{1, \left|\frac{3B_1(1+2 \a)}{4(1+\a )^2}(\mu -2)+\frac{B_2}{B_1}\right|\right\}.\]
\end{theorem}

\begin{proof}
In view of euqations (\ref{aB}) and (\ref{Aa}), we get
\begin{align*}
|A_3- \mu A^2_2|
&= \left|\dfrac{-B_1}{6(1+2\a )}p_2+\left(\left(\dfrac{B^2_1}{8(1+\a )^2}+\dfrac{B_1-B_2}{12(1+2\a )}\right) - \mu \frac{B^2_1}{16(1+\a)^2}\right)p^2_1\right|\\
&=\left|\dfrac{B_1}{6(1+2\a )}\left(p_2+\left(\frac{-3B_1(1+2\a )}{4(1+\a)^2}+\frac{B_2}{2B_1}- \frac{1}{2}+\mu \frac{3B_1(1+2\a)}{8(1+\a)^2}\right)p^2_1\right) \right|\\
&= \left|\dfrac{B_1}{6(1+2\a )}(p_2-\nu p^2_1)\right|,
\end{align*}

where $\nu = \dfrac{3B_1(1+2\a )}{8(1+\a)^2}\left(2-\mu\right)-\dfrac{B_2}{2B_1}+ \dfrac{1}{2}$. By Lemma \ref{l4}, we get the required result.
\end{proof}

\begin{theorem} \label{HD}
Let $\alpha \in \C$ such that $\RE \alpha \ge 0$ and the function $\phi$ defined by (\ref{phiz}). Suppose function $f$ in  $\r(\a , \phi)$ and $f^{-1}(\omega) = \om + \sum\limits_{n=2}^{\infty}A_n \om^n $ for all $\om$ in some neighbourhood of the origin.
\begin{itemize}
\item[(i)] If $B_1,B_2$ and $B_3$ satisfy the conditions
\[4d_2 \le d_3,\quad d_1 \le \frac{B_1|1+\a |}{9|(1+2\a )^2|},\]
then
\[|A_2 A_4-A^2_3| \le\frac{B^2_1}{9|(1+2\a )^2|}. \]

\item[(ii)] If $B_1,B_2$ and $B_3$ satisfy the conditions
\[4d_2 \ge d_3,\quad d_1 - \frac{d_2}{2} - \frac{B_1}{16|1+3\a |} \ge 0\]
or
\[4d_2 \le d_3,\quad d_1 \ge \frac{B_1|1+\a |}{9|(1+2\a )^2|},\]

then
\[|A_2 A_4-A^2_3| \le\frac{B_1 d_1}{|1+\a |}. \]

\item[(iii)] If $B_1,B_2$ and $B_3$ satisfy the conditions
\[4d_1 > d_3, d_1 - \frac{d_2}{2} - \frac{B_1}{16|1+3\a |} \le 0,\]

then
\[|A_2 A_4-A^2_3| \le\dfrac{\dfrac{B_1}{16|1+\a |}\left(\dfrac{16B_1|1+\a |}{9|(1+2 \a )^2|}d_1 - \dfrac{B_1 }{|1+3 \a |}d_2 - 4 d^2_2 - \dfrac{B^2_1}{16|(1+3 \a)^2|}\right)}{d_1-d_2- \dfrac{B_1}{8|1+3\a |} + \dfrac{B_1|1+ \a |}{9|(1+2 \a)^2|}}, \]
\end{itemize}
where
\begin{align*}
d_1 &= \frac{B^3_1}{16|(1+\a )^3|} + \frac{B^2_2 |1+\a |}{9B_1|(1+2 \a )^2|}+\frac{B_1|B_2|}{12|1+\a ||1+2 \a |} +\frac{|B_3|}{8|1+3 \a|}\\
d_2 &= \frac{B^2_1}{12|1+\a ||1+2\a |} +\frac{|B_2|}{4 |1+3\a |} +\frac{2|B_2||1+\a |}{9|(1+2\a)^2 |}\\
\text{and}\\
d_3 &= \frac{8B_1|1+\a |}{9|(1+2\a)^2 |}-\frac{B_1}{2|1+3\a |}.
\end{align*}

\end{theorem}
In the proof of this result, the following lemma is needed.
\begin{lemma} \cite [Lemma 2, p.254] {LIB1}\label{l3}
If $p(z) = 1+ \sum\limits_{n=1}^{\infty}p_n z^n \in \mathcal{P}$, then
\begin{align*}
2p_2 &= p^2_1+x(4-p^2_1)\\
4p_3 &=p^3_1+2p_1(4-p^2_1)x-p_1(4-p^2_1)x^2 +2(4-p^2_1)(1-|x|^2)z
\end{align*}
for some $x$ and $z$ such that $|x| \le 1$ and $|z|\le1$.
\end{lemma}

\begin{remark}
For real numbers $P$, $Q$ and $R$,  a standard computation gives
\begin{equation}
\label{QE} \max_{0 \le t \le4}(P t^2+Qt+R) =\left\{ \begin{array}{lll}
                R, & Q \le 0, P \le-Q/4 \\
                  16P+4Q+R , & Q\ge 0 ,P\ge -Q/8 \,\mbox{or}\,Q\le0,P\ge-Q/4\\
                  \dfrac{4PR-Q^2}{4P},& Q>0, P\le -Q/8
                             \end{array}
                       \right.
                       \end{equation}
\end{remark}
\begin{proof} [Proof of Theorem \ref{HD} ]
It follows from equations (\ref{aB}) and (\ref{Aa}) that
\begin{align*}
A_2A_4-A^2_3 &= \dfrac{-B_1p_1}{4(1+\a )}\left(\dfrac{-B_1}{8(1+3 \a )}p_3+\left(\dfrac{5B^2_1}{24(1+\a )(1+2\a )}+\dfrac{B_1-B_2}{8(1+3\a )}\right)p_1p_2 \right.\\
&\left. +\left(\dfrac{-5B^3_1}{6(1+\a )^3}-\dfrac{5B_1 (B_2-B_1)}{48(1+\a )(1+2\a )}-\dfrac{B_1-2B_2+B_3}{32(1+3\a )}\right)p^3_1\right)\\
&-\left(\left(\dfrac{B^2_1}{8(1+\a )^2}+\dfrac{(B_1-B_2)}{12(1+2\a )}\right)p^2_1-\dfrac{B_1 p_2}{6(1+2\a )}+\right)^2\\
&=p^4_1\left(\frac{B^4_1}{256(1+\a )^4} + \frac{B^2_1(B_1-B_2)}{192(1+\a )^2(1+2 \a)}+\frac{B_1(B_1-2B_2+B_3)}{128(1+\a) (1+3 \a )}\right.\\
 & \left.- \frac{(B_1-B_2)^2}{144(1+2 \a )^2} \right)+\frac{B^2_1 p_1p_3}{32(1+\a )(1+3\a )} - \frac{B^2_1 p^2_2}{36(1+2 \a )^2}\\
&+p^2_1p_2\left(\frac{-B^3_1}{96(1+\a )^2(1+2 \a )}-\frac{B_1(B_2-B_1)}{32(1+\a)(1+3\a )}+\frac{B_1(B_1-B_2)}{36(1+2\a )^2}\right).
\end{align*}
In view of Lemma \ref{l3}, we obtain
\begin{align*}
A_2A_4-A^2_3 &= \frac{B_1}{16(1+\a )} \left[p^4_1\left(\frac{B^3_1}{16(1+\a )^3} - \frac{B^2_2(1+\a )}{9B_1(1+2 \a)^2} - \frac{B_1B_2}{12(1+\a ) (1+2 \a)}+ \frac{B_3}{8(1+3 \a)}\right)\right.\\
&+2p^2_1(4-p^2_1)x \left(-\frac{B^2_1}{24(1+\a )(1+2 \a)}+\frac{B_2}{8(1+3\a)}-\frac{B_2(1+\a )}{9(1+2 \a)^2}\right)\\
&\left.+(4-p^2_1)x^2p^2_1\left(\frac{-B_1}{8(1+3\a)} \right) -(4-p^2_1)^2x^2\left(\frac{B_1(1+\a )}{9(1+2\a)^2} \right)\right.\\
&\left.+2p_1(4-p^2_1)(1-|x|^2)z\frac{B_1}{8(1+3\a)}\right].
\end{align*}
Since $|p_1| \le 2$ and it can be assumed that $p_1 >0$ and thus we get that $p_1 \in[0,2]$. Letting $p_1 = p$ and $|x| = \gamma$ in the above expression, we get
\begin{align*}
|A_2A_4-A^2_3| & \le \frac{B_1}{16|1+\a |} \left[p^4\left(\frac{B^3_1}{16|(1+\a )^3|} + \frac{B^2_2|1+\a |}{9B_1|(1+2 \a)^2|} + \frac{B_1|B_2|}{12|1+\a | |1+2 \a|}+ \frac{|B_3|}{8|1+3 \a|}\right)\right.\\
&+2p^2(4-p^2) \gamma\left(\frac{B^2_1}{24|1+\a ||1+2 \a|}+\frac{|B_2|}{8|1+3\a|}+\frac{|B_2||1+\a |}{9|(1+2 \a)^2|}\right)\\
&\left.+(4-p^2)\gamma^2p^2\left(\frac{B_1}{8|1+3\a|} \right) +(4-p^2)^2 \gamma^2\left(\frac{B_1|1+\a| }{9|(1+2\a)^2|} \right)\right.\\
&\left.+2p(4-p^2)(1-\gamma^2)|z|\frac{B_1}{8|1+3\a|}\right].\\
&=: G(p,\gamma).
\end{align*}
Let $p$ be fixed. Using the first derivative test in the region $\Omega = \{(p,\gamma): 0 \le p \le 2, 0 \le \gamma \le 1\}$ we get that $G(p,\gamma)$ is an increasing function of $\gamma$ where $\gamma \in [0,1]$. Thus for fixed $p \in [0,2]$, we obtain

$\max_{0\le \gamma \le 1}G(p,\gamma) = G(p,1) =:F(p)$,
where
\begin{align*}
F(p) &=\frac{B_1}{16|1+\a |}\left[ p^4\left(\frac{B^3_1}{16|(1+\a )^3|}+\frac{B^2_2|1+\a |}{9B_1|(1+2 \a)^2|}+\frac{B_1|B_2|}{12|1+\a | |1+2 \a|}+\frac{|B_3|}{8|1+3 \a|}\right.\right.\\
&\left.-\frac{B^2_1}{12|1+\a ||1+2 \a|}-\frac{|B_2|}{4|1+3\a|}-\frac{2|B_2||1+\a |}{9|(1+2 \a)^2|}-\frac{B_1}{8|1+3\a|}+\frac{B_1|1+\a |}{9|(1+2 \a)^2|}\right)\\
&+p^2\left(\frac{B^2_1}{3|1+\a ||1+2 \a|}+\frac{|B_2|}{|1+3\a|}+\frac{8|B_2||1+\a |}{9|(1+2 \a)^2|}+\frac{B_1}{2|1+3\a|}-\frac{8B_1|1+\a |}{9|(1+2 \a)^2|}\right)\\
&\left.+\frac{16B_1|1+\a |}{9|(1+2 \a)^2|}\right].
\end{align*}
Let
\begin{align*}
P&=\frac{B_1}{16|1+\a |}\left[\left(\frac{B^3_1}{16|(1+\a )^3|}+\frac{B^2_2|1+\a |}{9B_1|(1+2 \a)^2|}+\frac{B_1|B_2|}{12|1+\a | |1+2 \a|}+\frac{|B_3|}{8|1+3 \a|}\right.\right)\\
&-\left.\left(\frac{B^2_1}{12|1+\a ||1+2 \a|}+\frac{|B_2|}{4|1+3\a|}+\frac{2|B_2||1+\a |}{9|(1+2 \a)^2|}\right)+\left(-\frac{B_1}{8|1+3\a|}+\frac{B_1|1+\a |}{9|(1+2 \a)^2|}\right)\right],\\
Q&=\frac{B_1}{16|1+\a |}\left[ 4\left(\frac{B^2_1}{12|1+\a ||1+2 \a|}+\frac{B_2}{4|1+3\a|}+\frac{2|B_2||1+\a |}{9|(1+2 \a)^2|}\right)\right.\\
&\left.-\left(\frac{8B_1|1+\a |}{9|(1+2 \a)^2|}-\frac{B_1}{2|1+3\a|}\right)\right],\\
R&=\frac{B^2_1}{9|(1+2\a )^2|}\quad \mbox{and} \quad t = p^2.
\end{align*}
Then $F(t) = Pt^2+Qt+R$. Using the inequality (\ref{QE}), we get the required result.
\end{proof}

\end{document}